\documentclass[11pt]{article}

\usepackage[margin=1.15in]{geometry}
\usepackage{amsmath,amssymb,amsthm,mathtools}
\usepackage{microtype}
\usepackage{enumitem}
\usepackage[hidelinks]{hyperref}

\newtheorem{theorem}{Theorem}[section]
\newtheorem{lemma}[theorem]{Lemma}
\newtheorem{corollary}[theorem]{Corollary}

\newcommand{\diam}{\operatorname{diam}}

\newcommand{\ZZ}{\mathbb Z}
\newcommand{\RR}{\mathbb R}
\newcommand{\QQ}{\mathbb Q}

\title{Distances in Planar Integral Point Sets}
\author{J\'ozsef Solymosi%
\thanks{Department of Mathematics, University of British Columbia, Vancouver, BC V6T 1Z2, Canada, and Óbuda University, Budapest, Hungary.
Email: \texttt{solymosi@math.ubc.ca}.}}
\date{}

\begin{document}
\maketitle

\begin{abstract}
We show that very small distances in a planar
integral point set are essentially one-dimensional.  Let \(P\) be a
non-collinear set of \(n\) points in the plane, all of whose pairwise distances
are integers. We prove that, for a sufficiently small $c>0$, at most one pair of points can determine a distance below $n^{c\log\log{n}}$ unless all such
short pairs are supported on a single line.  We also give a
construction showing that the line-supported
alternative is necessary: there are arbitrarily large non-collinear integral
point sets with one off-line point and with first two distinct distances
\(2\) and \(4\). We conjecture that in the general case (no three points are collinear)
even the smallest distance should be large, at least $n^{c\log\log{n}}$, however we can prove a linear lower bound only.
\end{abstract}

\section{Introduction and main statement}

A finite set \(P\subset \RR^2\) is called an \emph{integral point set} if
\(|pq|\in\ZZ\) for every pair of distinct points \(p,q\in P\).  The classical
theorem of Anning and Erd\H{o}s says that an infinite planar integral point set is
collinear \cite{AnningErdos}.  Quantitatively, Greenfeld, Iliopoulou and Peluse
recently proved a near-optimal lower bound for the diameter of a non-collinear
integral point set \cite{GIP}.  We shall use the following consequence as a
black box.

\begin{theorem}[GIP diameter lower bound, black box]\label{thm:gip}
There are absolute constants \(c_{\rm GIP}>0\) and \(n_{\rm GIP}\) such that
for every non-collinear integral point set \(P\subset \RR^2\) with
\(|P|=n\ge n_{\rm GIP}\),
\[
        \diam(P) \ge n^{c_{\rm GIP}\log\log n}.
\]
\end{theorem}

For \(M\ge 1\), define the \emph{\(M\)-short-distance graph} \(G_M(P)\) on
vertex set \(P\) by joining two points whenever their distance is at most \(M\).
The result below says that, well below the GIP diameter scale, this graph has
only the obvious one-dimensional obstruction.

\begin{theorem}[Very short pairs are line-supported]\label{thm:main}
There are absolute constants \(\eta>0\) and \(n_0\) such that the following
holds.  Let \(P\subset \RR^2\) be a non-collinear integral point set with
\(|P|=n\ge n_0\), and let
\[
        1\le M\le n^{\eta\log\log n}.
\]
Then either \(G_M(P)\) has at most one edge, or all edges of \(G_M(P)\) are
contained in a single line.
\end{theorem}

Equivalently, if two different pairs of points determine distances at most
\(M\), and if these two pairs are not supported on one line, then
\(M\ge n^{\Omega(\log\log n)}\).  In particular, for every function
\(M(n)=n^{o(\log\log n)}\), the conclusion of Theorem~\ref{thm:main} holds for
all sufficiently large \(n\).

The proof is elementary after the use of Theorem~\ref{thm:gip}.
The first ingredient is a three-anchor localization lemma.  It is essentially
the usual hyperbola-intersection argument of Erd\H{o}s \cite{Erdos1945}, but we write it arithmetically after
putting the integral point set in a quadratic coordinate lattice.  The second
ingredient is a Gram-determinant calculation showing that two short pairs in
non-collinear position have polynomially bounded cross-distances.  The final
section gives a concrete family with one point off a line and
arbitrarily many points on the line, for which \(d_1=2\) and \(d_2=4\).  This
shows that the line-supported alternative in Theorem~\ref{thm:main} cannot be
removed.

Throughout the paper, all implicit constants are absolute and are not optimized.
All logarithms are natural.

\section{Three anchors and bounded triple intersections}

We begin with a small divisibility fact that will be used twice.

\begin{lemma}[Integer root bound]\label{lem:root}
Let
\[
        f(T)=a_dT^d+a_{d-1}T^{d-1}+\cdots+a_0\in\ZZ[T]
\]
be a nonzero polynomial of degree at most \(D\), and suppose that
\(|a_i|\le H\) for all \(i\).  If \(R\in\ZZ\) and \(f(R)=0\), then
\[
        |R|\le H.
\]
\end{lemma}

\begin{proof}
If \(R=0\), there is nothing to prove.  Otherwise let \(j\) be the smallest
index such that \(a_j\ne0\).  Then
\[
        f(T)=T^jg(T),
        \qquad
        g(T)=a_j+a_{j+1}T+\cdots+a_dT^{d-j},
\]
with \(g(0)=a_j\ne0\).  Since \(f(R)=0\) and \(R\ne0\), we have \(g(R)=0\).
Thus
\[
        a_j=-R\bigl(a_{j+1}+a_{j+2}R+\cdots+a_dR^{d-j-1}\bigr),
\]
so \(R\mid a_j\).  Hence \(|R|\le |a_j|\le H\).
\end{proof}

\begin{lemma}[Quadratic coordinate normalization]\label{lem:normalization}
Let \(P\) be an integral point set, and let \(A,B,C\in P\) be non-collinear.
After an isometry we may write
\[
        A=(0,0),\qquad B=(d,0),\qquad
        C=\left(\frac{U}{L},\frac{\sqrt{k}V}{L}\right),
        \qquad L=2d,
\]
where \(d=|AB|\in\ZZ_{>0}\), \(U,V\in\ZZ\), \(V\ne0\), and \(k\) is a positive
squarefree integer.  Moreover every point \(X\in P\) has coordinates
\[
        X=(x,\sqrt{k}y)
        \qquad\text{with }x,y\in\QQ.
\]
If \(|AB|,|AC|,|BC|\le a\), then
\[
        L\le 2a,
        \qquad |U|\le 2a^2,
        \qquad kV^2\le 4a^4.
\]
\end{lemma}

\begin{proof}
Place \(A=(0,0)\) and \(B=(d,0)\).  If \(c=|AC|\) and \(e=|BC|\), then
\[
        C_x=\frac{d^2+c^2-e^2}{2d}.
\]
Thus \(C_x=U/L\) with \(L=2d\) and \(U\in\ZZ\).  Since
\[
        C_y^2=c^2-\frac{U^2}{L^2},
\]
we may write \(C_y=\sqrt{k}V/L\), with \(k\) squarefree and \(V\in\ZZ\).  The
non-collinearity of \(A,B,C\) gives \(V\ne0\).

Now let \(X=(x,Y)\in P\).  Since \(|XA|\) and \(|XB|\) are integers, subtracting
squared distances gives
\[
        2dx=|XA|^2+d^2-|XB|^2\in\ZZ,
\]
so \(x\in\QQ\).  Since \(|XC|\) is also integral,
\[
        2C\cdot X=|XA|^2+|AC|^2-|XC|^2\in\ZZ.
\]
The term \((U/L)x\) is rational, so \((\sqrt{k}V/L)Y\) is rational.  As
\(V\ne0\), this gives \(Y=\sqrt{k}y\) with \(y\in\QQ\).

Finally, if the triangle has longest side at most \(a\), then \(L=2d\le2a\).
Also \(|U|/L=|C_x|\le |AC|\le a\), so \(|U|\le2a^2\).  Similarly,
\[
        \frac{kV^2}{L^2}=C_y^2\le |AC|^2\le a^2,
\]
whence \(kV^2\le4a^4\).
\end{proof}

The next lemma is the promised arithmetic version of the triple-intersection
bound.  The two integer differences
\(|XA|-|XB|\) and \(|XA|-|XC|\) select two hyperbola levels.  For fixed levels,
the possible radius \(|XA|\) is an integer root of a nonzero quadratic
polynomial with polynomially bounded coefficients.

\begin{lemma}[Three-anchor localization]\label{lem:threeanchors}
Let \(A,B,C\) be a non-collinear triangle in an integral point set, and suppose
that its longest side is at most \(a\), with \(a\ge2\).  Then every point \(X\)
at integral distance from \(A,B,C\) satisfies
\[
        |XA|\ll a^8.
\]
Consequently, all points at integral distance from \(A,B,C\) lie in a disk of
radius \(O(a^8)\).  Moreover, there are \(O(a^2)\) such points.
\end{lemma}

\begin{proof}
Use Lemma~\ref{lem:normalization} and write
\[
        A=(0,0),\qquad B=(d,0),\qquad
        C=\left(\frac{U}{L},\frac{\sqrt{k}V}{L}\right),
        \qquad L=2d.
\]
Let
\[
        X=(x,\sqrt{k}y),
        \qquad R=|XA|.
\]
Then \(R\in\ZZ\).  Put
\[
        m=|XA|-|XB|,
        \qquad n=|XA|-|XC|.
\]
The triangle inequality gives
\[
        |m|\le |AB|=d\le a,
        \qquad |n|\le |AC|\le a.
\]
Let \(c=|AC|\), and define
\[
        \alpha=d^2-m^2,
        \qquad \beta=c^2-n^2.
\]
Then \(\alpha,\beta\in\ZZ\), and \(|\alpha|,|\beta|\le a^2\).

Subtracting the squared equations for the distances from \(X\) to \(A\) and
\(B\), and using \(|XB|=R-m\), gives
\[
        (R-m)^2=R^2-2dx+d^2.
\]
Hence
\begin{equation}\label{eq:Lx}
        Lx=2mR+\alpha .
\end{equation}
Similarly, since \(|XC|=R-n\),
\[
        (R-n)^2=R^2-2C\cdot X+c^2,
\]
so
\[
        2C\cdot X=2nR+\beta.
\]
Using the displayed coordinates for \(C\) and \(X\), this is
\[
        \frac{2}{L}(Ux+kVy)=2nR+\beta.
\]
Multiplying by \(L^2\) and substituting \eqref{eq:Lx}, we obtain
\begin{equation}\label{eq:Ly}
        2kV(Ly)=A_1R+B_1,
\end{equation}
where
\[
        A_1=2nL^2-4mU,
        \qquad
        B_1=\beta L^2-2U\alpha .
\]
Now use \(R^2=x^2+ky^2\).  Multiplying by \(L^2\) gives
\[
        L^2R^2=(Lx)^2+k(Ly)^2.
\]
Combining this with \eqref{eq:Lx} and \eqref{eq:Ly}, we get
\begin{equation}\label{eq:triplepoly}
        4kV^2\left(L^2R^2-(2mR+\alpha)^2\right)
        -(A_1R+B_1)^2=0.
\end{equation}
Thus \(R\) is an integer root of the polynomial
\[
        F(T)=4kV^2\left(L^2T^2-(2mT+\alpha)^2\right)-(A_1T+B_1)^2.
\]

We claim that \(F\) is not identically zero.  For fixed \(m,n\), the two
linear equations above determine \(X\) as an affine function of \(R\), say
\[
        X=X_0+Re.
\]
If \(F\equiv0\), then \(|X_0+Re|^2=R^2\) as an identity in \(R\).  Hence
\[
        |e|=1,
        \qquad e\cdot X_0=0,
        \qquad |X_0|=0.
\]
Thus \(X_0=0\).  Plugging \(R=0\) into the two scalar-product equations gives
\[
        d^2-m^2=0,
        \qquad c^2-n^2=0.
\]
Therefore \(|m|=|AB|\) and \(|n|=|AC|\).  The linear part gives
\[
        B\cdot e=m,
        \qquad C\cdot e=n.
\]
Since \(|e|=1\), equality holds in Cauchy's inequality for both \(B\) and
\(C\).  Hence \(B\) and \(C\) are parallel, contradicting the non-collinearity
of \(A,B,C\).  Therefore \(F\not\equiv0\).

It remains to bound its coefficients.  From Lemma~\ref{lem:normalization},
\[
        L\le2a,
        \qquad |U|\le2a^2,
        \qquad kV^2\le4a^4.
\]
Also \(|m|,|n|\le a\) and \(|\alpha|,|\beta|\le a^2\).  Therefore
\[
        |A_1|\ll a^3,
        \qquad |B_1|\ll a^4.
\]
It follows from the formula for \(F\) that every coefficient of \(F\) is
\(O(a^8)\).  Since \(F\in\ZZ[T]\) is nonzero and \(R\in\ZZ\) is a root,
Lemma~\ref{lem:root} gives
\[
        R=|XA|\ll a^8.
\]
This proves the disk bound.

For the counting statement, there are at most \((2a+1)^2=O(a^2)\) choices for
\((m,n)\).  For each fixed pair \((m,n)\), the nonzero polynomial \(F\) has
degree at most \(2\), so there are at most two possible values of \(R\).  Once
\(m,n,R\) are fixed, the two linear equations determine \(X\) uniquely.  Hence
there are \(O(a^2)\) possible points \(X\).
\end{proof}

\section{Two short pairs}

We next prove that two short pairs in genuinely two-dimensional position cannot
be far apart.  This is the only place where we need four points.  The proof is
a Gram-determinant calculation in which one long cross-distance is an integer
root of an integer polynomial with polynomially bounded coefficients.

\begin{lemma}[Polynomial separation of two short pairs]\label{lem:twopairs}
Let \(A,B,C,D\) be four distinct points in the plane whose six mutual distances
are integers.  Suppose
\[
        |AB|=d,
        \qquad |CD|=e,
        \qquad 1\le d,e\le M.
\]
If \(A,B,C,D\) are not all collinear, then
\[
        \max\{|AC|,|AD|,|BC|,|BD|\}\ll M^6.
\]
\end{lemma}

\begin{proof}
Put
\[
        R=|AC|.
\]
Define the integer differences
\[
        u=|AC|-|BC|,
        \qquad
        v=|AD|-|BD|,
        \qquad
        h=|AD|-|AC|.
\]
By the triangle inequality,
\[
        |u|\le d,
        \qquad |v|\le d,
        \qquad |h|\le e.
\]
Moreover
\[
        |AD|=R+h,
        \qquad |BC|=R-u,
        \qquad |BD|=R+h-v.
\]

Use \(A\) as the origin.  The three vectors \(B-A,C-A,D-A\) lie in the plane,
so their Gram determinant is zero.  In terms of the six distances, the Gram
matrix is
\[
\begin{pmatrix}
        d^2
        &
        uR+\dfrac{d^2-u^2}{2}
        &
        v(R+h)+\dfrac{d^2-v^2}{2}
        \\[6pt]
        uR+\dfrac{d^2-u^2}{2}
        &
        R^2
        &
        R^2+hR+\dfrac{h^2-e^2}{2}
        \\[6pt]
        v(R+h)+\dfrac{d^2-v^2}{2}
        &
        R^2+hR+\dfrac{h^2-e^2}{2}
        &
        (R+h)^2
\end{pmatrix}.
\]
Thus \(R\) is a root of the polynomial
\[
\begin{aligned}
        F(T)=\det
\begin{pmatrix}
        d^2
        &
        uT+\dfrac{d^2-u^2}{2}
        &
        v(T+h)+\dfrac{d^2-v^2}{2}
        \\[6pt]
        uT+\dfrac{d^2-u^2}{2}
        &
        T^2
        &
        T^2+hT+\dfrac{h^2-e^2}{2}
        \\[6pt]
        v(T+h)+\dfrac{d^2-v^2}{2}
        &
        T^2+hT+\dfrac{h^2-e^2}{2}
        &
        (T+h)^2
\end{pmatrix}.
\end{aligned}
\]
Multiplying by \(8\), we obtain a polynomial in \(\ZZ[T]\).  A direct
expansion shows that \(8F(T)\) has degree at most \(4\), and every coefficient
is \(O(M^6)\); indeed every coefficient is an integral polynomial in
\(d,e,u,v,h\) of total degree at most \(6\), and these five parameters are
\(O(M)\).

We now show that \(F\) is not identically zero unless the four points are
collinear.  The coefficient of \(T^4\) in \(F(T)\) is
\[
        -(u-v)^2.
\]
Hence \(F\equiv0\) implies \(u=v\).  Substituting \(u=v\), the coefficient of
\(T^2\) becomes
\[
        (d^2-u^2)(e^2-h^2).
\]
Thus either \(|u|=d\), or \(|h|=e\).

If \(|u|=d\), then
\[
        \bigl||AC|-|BC|\bigr|=|AB|,
\]
so equality holds in the triangle inequality for the triangle \(ABC\).  Hence
\(A,B,C\) are collinear.  Since \(u=v\), we also have
\[
        \bigl||AD|-|BD|\bigr|=|AB|,
\]
so \(A,B,D\) are collinear.  Therefore \(A,B,C,D\) are all collinear.

If \(|h|=e\), then
\[
        \bigl||AD|-|AC|\bigr|=|CD|,
\]
so \(A,C,D\) are collinear.  Also, since \(u=v\),
\[
        |BD|-|BC|=(R+h-v)-(R-u)=h,
\]
so
\[
        \bigl||BD|-|BC|\bigr|=|CD|.
\]
Thus \(B,C,D\) are collinear.  Again all four points are collinear.  This
contradicts the hypothesis, and so \(F\not\equiv0\).

Since \(8F\in\ZZ[T]\) is nonzero, its coefficients are \(O(M^6)\), and
\(R\in\ZZ\) is a root, Lemma~\ref{lem:root} gives
\[
        |AC|=R\ll M^6.
\]
The remaining cross-distances satisfy
\[
        |AD|=R+h,
        \qquad |BC|=R-u,
        \qquad |BD|=R+h-v,
\]
with \(|u|,|v|,|h|\le M\).  Hence all four cross-distances are \(O(M^6)\).
\end{proof}

\begin{corollary}[Two non-collinear short pairs give small anchors]\label{cor:smallanchors}
Let \(P\) be an integral point set.  Suppose that \(P\) contains two distinct
pairs \(AB\) and \(CD\) of length at most \(M\).  If these two pairs are not
contained in a single line, then \(P\) contains three non-collinear points whose
mutual distances are all \(O(M^6)\).
\end{corollary}

\begin{proof}
If the two pairs share an endpoint, say the pairs are \(AB\) and \(AC\), and
\(A,B,C\) are not collinear, then
\[
        |BC|\le |BA|+|AC|\le 2M,
\]
so \(A,B,C\) themselves are suitable anchors.

If the two pairs have four distinct endpoints, and the four endpoints are not
all collinear, then Lemma~\ref{lem:twopairs} gives diameter \(O(M^6)\) for the
four endpoints.  Since they are not all collinear, some three of them are
non-collinear, and these three have mutual distances \(O(M^6)\).
\end{proof}

\section{Proof of the main theorem}

We now prove Theorem~\ref{thm:main}.  Let \(C_1,C_2\) be absolute constants such
that Lemma~\ref{lem:twopairs} and Lemma~\ref{lem:threeanchors} imply the
following: if \(P\) contains two pairs of length at most \(M\) not contained in a
single line, then \(P\) contains three non-collinear anchors of mutual distances
at most \(C_1M^6\), and consequently
\[
        \diam(P)\le C_2M^{48}.          \tag{4.1}
\]
Indeed, every point of \(P\) is at integral distance from the three anchors, so
Lemma~\ref{lem:threeanchors}, applied with \(a=C_1M^6\), puts all of \(P\) in a
disk of radius \(O(M^{48})\).

Choose
\[
        0<\eta<\frac{c_{\rm GIP}}{100}.
\]
Increasing \(n_0\), if necessary, we may assume that for all \(n\ge n_0\),
\[
        C_2 n^{48\eta\log\log n}<n^{c_{\rm GIP}\log\log n}.
\]
Now let \(P\) be a non-collinear integral point set with \(|P|=n\ge n_0\), and
let \(M\le n^{\eta\log\log n}\).

Suppose that \(G_M(P)\) has two edges not contained in a single line.  By the
preceding paragraph,
\[
        \diam(P)
        \le C_2M^{48}
        \le C_2n^{48\eta\log\log n}
        < n^{c_{\rm GIP}\log\log n},
\]
contradicting the GIP lower bound, Theorem~\ref{thm:gip}.  Therefore every two
edges of \(G_M(P)\) are contained in a single line.  If \(G_M(P)\) has at most
one edge, we are done.  Otherwise choose one edge \(pq\) of \(G_M(P)\), and let
\(\ell\) be the line through \(p\) and \(q\).  Every other edge must be contained
in a line together with \(pq\), hence must also lie on \(\ell\).  Thus all edges
of \(G_M(P)\) are contained in one line.

This proves Theorem~\ref{thm:main}.

\section{\texorpdfstring{An almost-collinear construction with \(d_1=2\) and \(d_2=4\)}{An almost-collinear construction with d1=2 and d2=4}}

Theorem~\ref{thm:main} cannot be strengthened by removing the line-supported
alternative.  We give an explicit family with one point off a line and
arbitrarily many points on the line, in which the first two distinct distances
are \(2\) and \(4\).

Start with a positive integer \(N\).  Put
\[
        q=(1,2\sqrt N).
\]
For every divisor \(a\mid N\), define a line point
\[
        p_a=\left(1+\frac Na-a,0\right).
\]
Then
\[
\begin{aligned}
        |q-p_a|^2
        &=\left(\frac Na-a\right)^2+4N  \\
        &=\left(\frac Na+a\right)^2.
\end{aligned}
\]
Hence
\[
        |q-p_a|=\frac Na+a\in\ZZ.
\]
Also all distances among the line points \(p_a\) are integers, since their
\(x\)-coordinates are integers.  Thus every finite subset of these points,
together with \(q\), is an integral point set.

We now choose \(N\) so that the line contains prescribed gaps \(2\) and \(4\),
while the number of available divisors tends to infinity.  Consider the Pell
family
\[
        R_j+U_j\sqrt3=(1+\sqrt3)(2+\sqrt3)^j,
        \qquad j=0,1,2,\ldots .
\]
Then
\[
        R_j^2-3U_j^2=-2,
\]
and \(R_j,U_j\) are odd.  Set
\[
        N_j=\frac{R_j^2-1}{4}=\frac{3(U_j^2-1)}{4}.
\]
The two divisors
\[
        a_1=\frac{R_j+1}{2},
        \qquad
        a_2=\frac{R_j-1}{2}
\]
produce the line points
\[
        x_{a_1}=1+\frac{N_j}{a_1}-a_1=0,
        \qquad
        x_{a_2}=1+\frac{N_j}{a_2}-a_2=2.
\]
Thus the line contains a gap of length \(2\).  Similarly, the two divisors
\[
        b_1=\frac{U_j+1}{2},
        \qquad
        b_2=\frac{U_j-1}{2}
\]
produce
\[
        x_{b_1}=U_j-1,
        \qquad
        x_{b_2}=U_j+3,
\]
so the line contains a gap of length \(4\).

It remains to see that we can make arbitrarily many line points.  The divisor
function \(\tau(N_j)\) is unbounded along this Pell family.  Indeed, fix any
finite set \(S\) of odd primes.  The recurrence
\[
        \binom{R_{j+1}}{U_{j+1}}
        =
        \begin{pmatrix}2&3\\1&2\end{pmatrix}
        \binom{R_j}{U_j}
\]
has an invertible transition matrix modulo every odd prime.  Hence, modulo each
\(p\in S\), the orbit \((R_j,U_j)\) is periodic and returns to
\((R_0,U_0)=(1,1)\) after some positive period.  Choosing \(j\) to be a common
multiple of these periods gives
\[
        R_j\equiv1\pmod p
        \qquad\text{for every }p\in S.
\]
Therefore every \(p\in S\) divides \(N_j=(R_j^2-1)/4\).  Since \(S\) is
arbitrary, the numbers \(N_j\) have arbitrarily many distinct prime divisors,
and so \(\tau(N_j)\to\infty\) along a subsequence.

For such an \(N_j\), let
\[
        X_j=\left\{1+\frac{N_j}{a}-a:\\ a\mid N_j\right\}\subset\ZZ.
\]
This set has \(\tau(N_j)\) elements and contains
\[
        0,
        \quad 2,
        \quad U_j-1,
        \quad U_j+3.
\]
For large \(j\), these four special points are distinct and have no mutual gap
\(1\) or \(3\), while they do have gaps \(2\) and \(4\).  Delete from
\(X_j\setminus\{0,2,U_j-1,U_j+3\}\) all points lying at distance \(1\) or \(3\)
from one of the four special points.  This removes only \(O(1)\) points.
On the remaining set, form the graph in which two integers are adjacent if
their difference is \(1\) or \(3\).  This graph has maximum degree at most
\(4\), and hence contains an independent set of size at least one fifth of its
number of vertices.  Adjoining the four special points to such an independent
set gives a subset
\[
        Y_j\subset X_j
\]
with \(|Y_j|\to\infty\), containing gaps \(2\) and \(4\), and containing no gap
\(1\) or \(3\).

Now define
\[
        P_j=\{q_j\}\cup\{(x,0):x\in Y_j\},
        \qquad q_j=(1,2\sqrt{N_j}).
\]
Then \(P_j\) is a non-collinear integral point set.  Its line-line distances
are positive integers with no values \(1\) or \(3\), and with values \(2\) and
\(4\) present.  The off-line distances are
\[
        |q_j-p_a|=\frac{N_j}{a}+a\ge 2\sqrt{N_j},
\]
which are larger than \(4\) for all large \(j\).  Therefore, for all large
\(j\), the first two distinct distances in \(P_j\) are exactly
\[
        d_1(P_j)=2,
        \qquad
        d_2(P_j)=4,
\]
while \(|P_j|\to\infty\).  This is the desired almost-collinear obstruction.

\section{The minimal distance}

In the previous sections we proved that for integral point sets with no three collinear points only one pair can determine a smaller distance.
On the other hand we conjecture that even the minimal distance is at least $n^{c\log{\log{n}}}$ for $n$-element integral sets (with some fixed $c>0$).
The best lower bound we can prove is that the minimal distance is at least $n/2$. We are using the fact that a hyperbola arc, far from the center 
can't hold a triangle with integer side-lengths. This idea was used in \cite{Soly1} for bounding the diameter of integral pointsets. The $n/2$ lower bound 
follows directly from the lemma below.

\begin{lemma}[One point per asymptote class]
Let \(P\) be an integral point set, and let \(p,q\in P\) realize the minimum
distance
\[
        |pq|=d.
\]
Assume that every distance in \(P\), except \(|pq|\), is larger than \(4d^2\).
Then every pair of centrally symmetric hyperbola arcs determined by \(p,q\)
contains at most one point of \(P\setminus\{p,q\}\).
\end{lemma}

\begin{proof}
Place
\[
        p=\left(-\frac d2,0\right),
        \qquad
        q=\left(\frac d2,0\right).
\]
For \(x\in P\setminus\{p,q\}\), put
\[
        r_x=|xp|,\qquad s_x=|xq|,
\]
and
\[
        k_x=r_x-s_x,\qquad z_x=r_x+s_x.
\]
Then \(k_x,z_x\in\mathbb Z\), \(z_x\equiv k_x\pmod 2\), and
\[
        |k_x|<d.
\]

For fixed \(k\), a signed arc is parametrized by
\[
        x(z,\sigma)
        =
        \left(
        \frac{kz}{2d},
        \sigma\frac{\sqrt{d^2-k^2}}{2d}\sqrt{z^2-d^2}
        \right),
        \qquad \sigma\in\{+1,-1\}.
\]
The centrally symmetric arc is
\[
        x(w,-\sigma)
        \quad\text{on the level }-k.
\]

Suppose first that two points lie on the same signed arc:
\[
        a=x(z,\sigma),\qquad b=x(z+2t,\sigma),
        \qquad t\in\mathbb Z_{\ge1}.
\]
Then
\[
        |bp|-|ap|=t.
\]
By strict triangle inequality in the triangle \(pab\),
\[
        |ab|>t.
\]
On the other hand, the standard estimate
\[
        \sqrt{u^2-d^2}=u-\frac{d^2}{u+\sqrt{u^2-d^2}}
\]
gives, since \(z>4d^2\),
\[
        |ab|<t+1.
\]
Thus
\[
        t<|ab|<t+1,
\]
which is impossible because both \(t\) and \(|ab|\) are integers.

Now suppose that one point lies on the arc \((k,\sigma)\) and the other on
the centrally symmetric arc \((-k,-\sigma)\):
\[
        a=x(z,\sigma),\qquad b=x(w,-\sigma).
\]
Let
\[
        S=|ap|+|bp|.
\]
Since both distances are integers,
\[
        S\in\mathbb Z.
\]
A direct calculation gives
\[
        S=\frac{z+w}{2}.
\]
Write
\[
        g(u)=\sqrt{u^2-d^2}.
\]
Then
\[
        |ab|^2
        =
        \left(\frac{k(z+w)}{2d}\right)^2
        +
        \left(
        \frac{\sqrt{d^2-k^2}}{2d}(g(z)+g(w))
        \right)^2 .
\]
Since \(g(u)<u\), we have
\[
        |ab|<\frac{z+w}{2}=S.
\]
Moreover,
\[
        u-g(u)=\frac{d^2}{u+g(u)}<\frac{d^2}{u}.
\]
Hence, using \(z,w>4d^2\),
\[
        S-|ab|
        <
        \frac{d^2}{z}+\frac{d^2}{w}
        <
        1.
\]
Therefore
\[
        S-1<|ab|<S.
\]
Again this is impossible because \(S\) and \(|ab|\) are integers.

Thus neither one arc nor its centrally symmetric partner can contain two
points in total.  This proves the lemma.
\end{proof}

\section{Acknowledgements} The author is supported by an NSERC Discovery Grant.

\end{document}